\documentclass[twoside]{amsart}
\usepackage{latexsym}
\usepackage{amssymb,amsmath,amsopn}
\usepackage[dvips]{graphicx}   
\usepackage{color,epsfig}      
\usepackage{fullpage}
\usepackage{setspace}
\usepackage{amsthm}
\usepackage{wrapfig}

\newtheorem{thm}{Theorem}

\newtheorem{lem}{Lemma}
\newtheorem{cor}{Corollary}

\theoremstyle{definition}
\newtheorem{defn}{Definition}
\newtheorem{rem}{Remark}
\newtheorem{ques}{Question}

\renewcommand{\Re}{\mathbb R}
\newcommand{\Sph}{\mathbb S}

\newcommand{\F}{\mathcal{F}}
\newcommand{\T}{\mathcal{T}}

\DeclareMathOperator{\inter}{int}
\DeclareMathOperator{\bd}{bd}

\DeclareMathOperator{\conv}{conv}

\begin{document}
\title[Decompositions of a polygon]{Decompositions of a polygon into centrally symmetric pieces}
\author[J. Frittmann and Z. L\'angi]{J\'ulia Frittmann and Zsolt L\'angi}
\address{J\'ulia Frittmann, Budapest University of Technology and Economics, Budapest, Egry J\'ozsef u. 1.,
Hungary, 1111}
\email{frittmannjuli@gmail.com}
\address{Zsolt L\'angi, Dept. of Geometry, Budapest University of
Technology and Economics, Budapest, Egry J\'ozsef u. 1., Hungary, 1111}
\email{zlangi@math.bme.hu}
\thanks{Partially supported by the J\'anos Bolyai Research Scholarship of the Hungarian Academy of Sciences}

\begin{abstract}
In this paper we deal with edge-to-edge, irreducible decompositions of a centrally symmetric convex $(2k)$-gon into centrally symmetric convex pieces.
We prove an upper bound on the number of these decompositions for any value of $k$, and characterize them for octagons.
\end{abstract}

\maketitle

\section{Introduction}

Tiling the plane or a given region has been of interest to mathematicians and artists as well, throughout the history of mankind.
In a large number of problems, the region to tile is often a polygon in the Euclidean plane $\Re^2$.
Without completeness, we list a few of these problems that have appeared in the literature.

In 1970, Monsky \cite{M70} showed that a square cannot be dissected into an odd number of triangles of equal areas. In a series of papers this problem, and result, was generalized for other figures and dimensions (cf. \cite{S89}, \cite{K89}, \cite{KS90}).
In recent years, another problem: the enumeration and characterization of the number of different dissections of a regular polygon with noncrossing diagonals, seems also to have been in the focus of research (cf. \cite{R78}, \cite{L95}, \cite{PS00}, \cite{T06}, \cite{HHNO09} or \cite{BR14});
in particular, \.Zak \cite{Z10} investigated the problem of identifying topologically equivalent dissections of this type.
Brooks, Smith, Stone and Tutte \cite{BSST40} examined the problem of dissecting a rectangle into squares of different sizes.
In \cite{KKU00}, the authors consider the problem of finding the minimum number of cuts to dissect a regular $m$-gon into a regular $n$-gon of the same area.
For more information on related questions, the author is referred to \cite{GR95}.

In 1997, G.Horv\'ath introduced the notion of \emph{irreducible} dissections of a centrally symmetric polygon into centrally symmetric convex polygons. He gave an example showing that for a centrally symmetric hexagon, there are infinitely many combinatorial types of such dissections. On the other hand, he proved that, up to combinatorial equivalence, there are exactly six irreducible, \emph{edge-to-edge} decompositions of this region.

The aim of this paper is twofold. In the first part, we show that, up to combinatorial equivalence, there are only finitely many irreducible, edge-to-edge decompositions of a centrally symmetric $(2k)$-gon for every value of $k$, and prove an upper bound on their number. In the second part we characterize the combinatorial classes of these decompositions for centrally symmetric convex octagons.

We start with some formal definitions.

\begin{defn}\label{defn:decomposition}
Let $P$ be a convex polygon. A family $\T$ of mutually nonoverlapping convex polygons covering $P$ is called a \emph{tiling} or \emph{decomposition} of $P$.
A decomposition $\T$ of $P$ is called \emph{edge-to-edge}, if for every edge $E$ of any member $F_i$ of $\T$, either $E$ is in $\bd P$, or there is some
$F_j \in \T$, where $i \neq j$, such that $E$ is an edge of $F_j$.
A decomposition $\T$ is called \emph{irreducible}, if $| \T | > 1$, and for any subfamily $\T'$ of $\T$ with the property that $\bigcup_{F_i \in \T'} F_i$ is convex, we have $| \T' | = 1$ or $| \T' | = | \T |$.
\end{defn}

It was observed in \cite{GHA97} (cf. also \cite{A58}) that only a centrally symmetric polygon can be decomposed into centrally symmetric pieces. Note that the elements of an edge-to-edge tiling, together with their edges and vertices, form a CW-decomposition. If the face lattices of the CW-decompositions induced by two edge-to-edge tilings $\T_1$ and $\T_2$ of the centrally symmetric $(2k)$-gons $P_1, P_2$ are isomorphic, we say that $\T_1$ and $\T_2$ are combinatorially equivalent. The \emph{combinatorial class} of a decomposition $\T$ is the equivalence class of $\T$ defined by this equivalence relation. Note that the same combinatorial classes are represented among the tilings of any two centrally symmetric $(2k)$-gons.

Our two main results are the following.

\begin{thm}\label{thm:upperbound}
Let $k \geq 4$. Then the number of the combinatorial classes of the irreducible, edge-to-edge decompositions of a centrally symmetric $(2k)$-gon into centrally symmetric parts is at most $\frac{6N (4N+1)!}{N! (3N+3)!}$, where $N = \left\lfloor \frac{2k^3(2k-3)^2}{\left( 2-\sqrt{2} \right) \pi^2 } \right\rfloor$.
\end{thm}

\begin{thm}\label{thm:main}
There are $111$ combinatorial classes of irreducible, edge-to-edge decompositions of a centrally symmetric convex octagon $P$. Furthermore, any irreducible, edge-to-edge decomposition of $P$ is combinatorially equivalent to one of the $111$ decompositions shown in the Appendix.
\end{thm}

In Sections~\ref{sec:upperbound} and \ref{sec:main}, we prove Theorems~\ref{thm:upperbound} and \ref{thm:main}, respectively. Finally, in Section~\ref{sec:remarks}
we collect additional remarks and propose some problems. For brevity, throughout this paper by a decomposition or tiling we mean an irreducible, edge-to-edge decomposition of a centrally symmetric polygon into centrally symmetric pieces.

\section{Proof of Theorem~\ref{thm:upperbound}}\label{sec:upperbound}

Let $P$ be a $(2k)$-gon, with $k \geq 4$. Since the combinatorial types of the possible tilings of any two centrally symmetric $(2k)$-gons are equal, we may assume that $P$ is regular, and that its sides are of unit length. Throughout this section, $P$ is a centrally symmetric $(2k)$-gon, and $\T$ is a tiling of $P$.
Furthermore, we denote by $\F$ the CW-complex induced by $\T$. The faces of $\F$ are called also \emph{tiles} of $\F$.

We start with two lemmas. 

\begin{lem}\label{lem:edgeclasses}
Let $E$ be an edge of $\F$. Then there is a unique sequence $E_1, E_2, \ldots, E_m$ of edges of $\F$, consisting of translates of $E$ satisfying the properties:
\begin{itemize}
\item $E_1$ and $E_m$ are contained in $\bd P$,
\item for $j=1,2,\ldots, m-1$, $E_j$ and $E_{j+1}$ belong to the same tile in $\F$.
\end{itemize}
\end{lem}

The proof of Lemma~\ref{lem:edgeclasses} is straightforward, and hence, we omit it, and note that the same property was observed in \cite{GHA97} as well.
In the following, for any edge $E$ we call this sequence the \emph{class} of $E$. Clearly, these classes define an equivalence relation on the set of edges of $\F$. Furthermore, note that increasing or decreasing the length of every member of any given class does not change the combinatorial type of $\F$. Hence,
without loss of generality, we may assume that any two parallel edges of $\F$ are of equal length.

\begin{cor}\label{cor:edgesareparallel}
Any edge of $\F$ is parallel to some side of $P$. Thus, any tile in a tiling of a $(2k)$-gon has at most $2k$ sides.
\end{cor}

\begin{cor}\label{cor:oppositeedges}
Let $E=[a,b]$ and $E'=[c,d]$ be opposite edges of $P$, in counterclockwise order. Assume that the edges of $\F$ in $E$ are $E_1, E_2, \ldots, E_s$ in counterclockwise order. Then the edges of $\F$ in $E'$ are the translates of $E_1 ,E_2, \ldots, E_s$ by $d-a$, in \emph{clockwise} order.
\end{cor}

\begin{lem}\label{lem:maxnumberofedges}
The complex $\F$ has at most $2k-3$ edges on any edge $E$ of $P$.
\end{lem}

\begin{proof}
Consider a Cartesian coordinate system such that the $y$-coordinate axis contains the edge $E=[a,b]$, the $y$-coordinate of $a$ is greater than that of $b$, and
$P$ is contained in the closed half plane $\{ x \geq 0 \}$.
Let $s$ be the number of edges of $\F$ that intersect, but are not contained in, $E$.
We label these edges in the following way:
\begin{itemize}
\item $a \in E_1$, $b \in E_s$, and $E_1, E_s \in \bd P$;
\item for any $i=2,3,\ldots,s-1$, the closed half plane, bounded by the line of $E_i$ and containing $E_1$, contains exactly the edges $E_1, E_2, \ldots, E_i$.
\end{itemize}
Let the slope of $E_i$ be denoted by $\mu_i$.

We show that the sequence $\mu_1, \mu_3, \ldots, \mu_s$ is decreasing, and no value is attained more than two times.
Indeed, the fact that $\mu_i \geq \mu_{i+1}$ follows from the fact that every tile is a centrally symmetric convex polygon. Furthermore,
if $\mu_i = \mu_{i+1}$, then $E_i$ and $E_{i+1}$ belongs to the same tile, which is, thus, the parallelogram with $E_i$ and $E_{i+1}$ as parallel sides.

Now, consider the case $\mu_{i-1}=\mu_i=\mu_{i+1}$.
Let $P_i$ and $P_{i+1}$ be the parallelograms defined by $E_{i-1}$ and $E_i$, and by $E_i$ and $E_{i+1}$, respectively. Then both $P_i$ and $P_{i+1}$ are
tiles of $\F$, $P_i \cup P_{i+1}$ is convex and is properly contained in $P$, which contradicts our assumption that $\T$ is irreducible.

We have shown that the sequence $\mu_1, \mu_3, \ldots, \mu_s$ is decreasing, and no value is attained more than two times.
Corollary~\ref{cor:edgesareparallel} implies that the slope of any edge of $\F$, not parallel to $E$ is $\tan \left( -\frac{\pi}{2} + \frac{i\pi}{k} \right)$ for some $i \in \{1,2,\ldots, k-1 \}$. This yields that $s \leq 2k-2$, from which it follows that $E$ is dissected into at most $s-1 \leq 2k-3$ nondegenerate segments.
\end{proof}

\begin{cor}\label{cor:atmosttwoedges}
For any edge $E$ of $P$ and direction $v \in \Sph^1$ not parallel to $E$, there are at most two edges of $\F$ intersecting $E$ and parallel to $v$.
\end{cor}

Now we prove Theorem~\ref{thm:upperbound}.

\begin{proof}[Proof of Theorem~\ref{thm:upperbound}]
First we estimate the number of edges of $\F$ from above.
Consider an edge $E$ of $\F$. Then, by Lemma~\ref{lem:edgeclasses}, $E$ is the translate of an edge of $\F$ in $\bd P$.
Since, according to our assumptions, any two parallel edges of $\F$ are of equal length, and any edge of $P$ is of unit length, and, by Lemma~\ref{lem:maxnumberofedges}, is dissected into at most $2k-3$ edges of $\F$, it follows that the length of $E$, and any other edge of $\F$, is at least $\frac{1}{2k-3}$.

Let $E'$ be an edge of $\F$, in the class of $E$, which is consecutive to $E$ in this class. Then there is a face $F$ of $\F$ in which $E$ and $E'$ are opposite sides. Since every angle of $F$ is a multiple of $\frac{\pi}{k}$, and each side of $F$ is of length at least $\frac{1}{2k-3}$, the distance of $E$ and $E'$ is at least $\frac{\sin \frac{\pi}{k}}{2k-3}$. Now, an elementary computation shows that the distance of any two parallel sides of $P$ is $\cot \frac{\pi}{2k}$. Thus, the number of edges in the class of $E$ is at most $\frac{\cot \frac{\pi}{2k}}{\frac{\sin \frac{\pi}{k}}{2k-3}} = \frac{2k-3}{2 \sin^2 \frac{\pi}{2k}}$.
Since there are at most $2k-3$ classes of edges parallel to any given pair of parallel sides of $P$, and there are $k$ pairs of parallel sides of $P$, the number $N$ of edges of $\F$ is at most $\frac{k(2k-3)^2}{2 \sin^2 \frac{\pi}{2k}}$.
Using the concavity of the function $x \mapsto \sin x$, and the inequality $\sin x \geq x \sin \frac{\pi}{8}$ for $x \in \left[ 0, \frac{\pi}{8} \right]$, we obtain that
\[
N \leq \frac{k(2k-3)^2}{2 \sin^2 \frac{\pi}{2k}} \leq \frac{2k^3(2k-3)^2}{\left( 2-\sqrt{2} \right) \pi^2 } .
\] 

Finally, observe that the vertices and edges of $\F$ form a loopless (and isthmusless) planar map. This map can be made \emph{rooted} by choosing a vertex $v$ of $\F$ as root, an edge $E$ starting at $v$, and the left or right-hand side of $E$ \cite{T63}. Clearly, the number of the combinatorial classes of loopless rooted maps with at most $N$ edges is an upper bound on the number of combinatorial classes of the tilings of $P$. By formula (6) of \cite{WL75}, the number $t_N$ of loopless (or isthmusless) rooted maps with $N$ edges is $t_N = \frac{6 (4N+1)!}{N! (3N+3)!}$.
Thus, the number of combinatorial classes of tilings is at most $\frac{6N (4N+1)!}{N! (3N+3)!}$.
\end{proof}

\section{Proof of Theorem~\ref{thm:main}}\label{sec:main}

Like in Section~\ref{sec:upperbound}, let $\T$ be a tiling of $P$, where $P$ is a regular octagon of unit side length, and let $\F$ be the CW-complex induced by $\F$.  We denote the sides of $P$ by $E_1, E_2, \ldots, E_8$, in counterclockwise order. For simplicity, we imagine that $E_1$ is horizontal, and, without loss of generality, we assume that any two parallel edges of $\F$ have equal lengths. The \emph{neighbors} of the side $E_i$ are $E_{i-1}$ and $E_{i+1}$.

By Lemma~\ref{lem:maxnumberofedges}, any edge of $P$ consists of at most five edges of $\F$.
For simplicity, if an edge $E$ of $P$ is dissected into $m$ edges of $\F$, where $m \in \{ 1,2,3,4,5 \}$, we say that the \emph{type} of $E$ is $m$.
Then, by Corollary~\ref{cor:oppositeedges}, the types of opposite edges of $P$ are the same.
If the types of $E_1$, $E_2$, $E_3$ and $E_4$ are $k_1$, $k_2$, $k_3$ and $k_4$, respectively, then we say that the type of $P$ is $k_1/k_2/k_3/k_4$.
In the proof, we may write about combinatorially equivalent subsets of $\T$. This concept can be defined analogously to the combinatorial equivalence of tilings.

First, we need two lemmas.

\begin{lem}\label{lem:norightangle}
Let $E'_1$ and $E'_2$ be two perpendicular edges of $\F$, with a common endpoint. Then there is a tile $F \in \F$ containing both $E'_1$ and $E'_2$. Furthermore, $F$ is either a rectangle or a hexagon.
\end{lem}

\begin{proof}
Let $p$ be the common endpoint of $E'_1$ and $E'_2$. We prove by contradiction, and assume that there is no tile containing both $E'_1$ and $E'_2$. Then there are two tiles $F_1$ and $F_2$, having a common edge $E'_3$, such that $E'_1 \subset F_1$, $E'_2 \subset F_2$, and the angles of $F_1$ and $F_2$ at $p$ are $\frac{\pi}{4}$.
Since every edge in $\F$ is parallel to some edge of $P$, it follows that $F_1$ and $F_2$ are parallelograms. Let the endpoint of $E'_3$, different from $p$, be $q$, and, for $i=1,2$, let $E''_i$ be the edge of $T_i$ parallel to $E_i$. Then, since the decomposition is irreducible, there is no edge of $\F$ containing $q$, but $E''_1$, $E''_2$ and $E'_3$.
Let $F_3$ be the tile containing $E''_1$ and $E''_2$. Then $F_1 \cup F_2 \cup F_3$ is either a parallelogram or a convex hexagon, contradicting the irreducibility of $\T$. Thus, $E'_1$ and $E'_2$ belong to the same tile $F$. The second statement follows from the observation that each side of $F$ is parallel to a side of $P$.
\end{proof}

\begin{figure}[ht]
\includegraphics[width=0.8\textwidth]{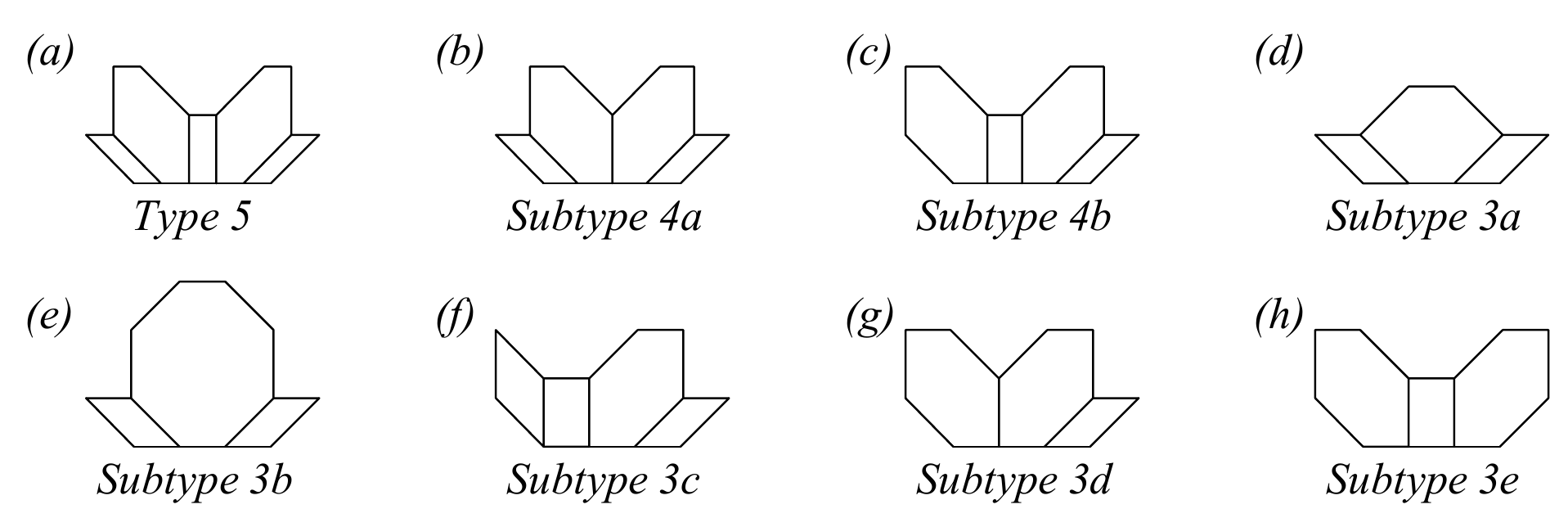}
\caption[]{Possible subtypes of an edge of $P$}
\label{fig:types}
\end{figure}

\begin{lem}\label{lem:types}
Let $E$ be an edge of $P$, of type $3$, $4$ or $5$. Then the collection of the tiles of $\F$, intersecting $E$, is combinatorially equivalent to one of the configurations in Figure~\ref{fig:types}.
\end{lem}

\begin{proof}
Without loss of generality, let $E=E_1$.

\emph{Case 1}, $E$ is type $5$. Let the edges of $\F$ in $E$ be denoted by $E'_1, E'_2, \ldots, E'_5$ in counterclockwise order. Then, by Corollary~\ref{cor:atmosttwoedges}, there are exactly six edges in $\F$ intersecting $E$ and not contained in it. Let them be $E''_1, E''_2, \ldots, E''_6$ where the indices are chosen according to the order of their intersection points with $E$, in counterclockwise direction. Then $E''_1 \subseteq E_8$ and $E''_6 \subseteq E_2$, $E''_2$ is parallel to $E_8$, $E''_5$ is parallel to $E_2$, and $E''_3$ and $E''_4$ are parallel to $E_3$.

Note that as every other edge in $\F$ is disjoint from $E$, for each of the triples $\{ E''_i, E'_i, E''_{i+1} \}$, where $i=1,2,\ldots,5$, there is a face of $\F$ containing all edges in it. Since every edge in $\F$ is parallel to an edge of $P$, the faces containing $E'_1$, $E'_3$ and $E'_5$ are parallelograms, whereas the ones containing $E'_2$ and $E'_4$ are hexagons. This finishes the proof in this case.

\begin{figure}[ht]
\includegraphics[width=0.25\textwidth]{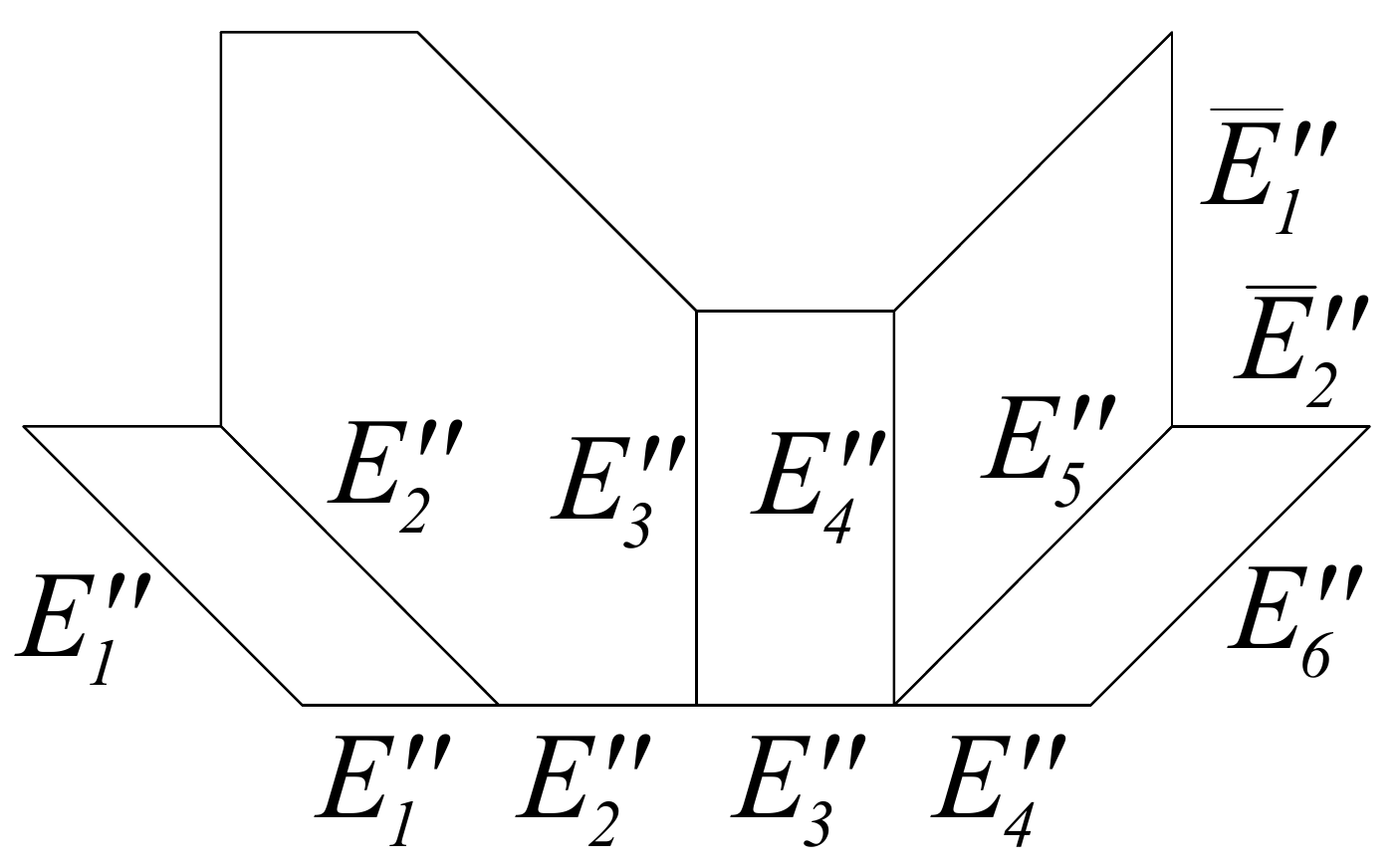}
\caption[]{An illustration for Case 2 in the proof of Lemma~\ref{lem:types}}
\label{fig:lemma3}
\end{figure}

\emph{Case 2}, $E$ is type $4$. Let the edges of $\F$ in $E$ be denoted by $E'_1, E'_2, E'_3, E'_4$, and the ones intersecting $E$ that are not contained in it, by $E''_1, \ldots, E''_u$, where, in labelling the edges we used the same convention as in the previous case. Note that $u=5$ or $u=6$, and in the latter case two edges meet on the edge $E$.

First, we consider the case that $u=6$. Then $E''_2$ is parallel to $E''_1 \subseteq E_8$, $E''_5$ is parallel to $E''_6 \subseteq E_2$, and $E''_3$ and $E''_4$ are parallel to $E_3$. Without loss of generality, we may assume that $E''_4$ and $E''_5$ meet at the common point of $E'_3$ and $E'_4$. Then, the tile containing $E''_4$ and $E''_5$ is a parallelogram $F_1$, and the tile containing $E'_4$ is a parallelogram $F_2$, as well. Let $\bar{E}_1$ denote the edge of $F_1$ parallel to $E''_4$ and, similarly, let $\bar{E}_2$ denote the edge of $F_2$ parallel to $E'_4$ (cf. Figure~\ref{fig:lemma3}). By Lemma~\ref{lem:norightangle}, there is a tile $F_3$ containing both $\bar{E}_1$ and $\bar{E}_2$, and $F_3$ is either a rectangle, or a hexagon. Nevertheless, in both cases $F_1 \cup F_2 \cup F_3$ is convex, which contradicts the condition that $\T$ is irreducible.

Now let $u=5$. Then we have one of the following.
\begin{enumerate}
\item[(i)] Only $E''_3$ is parallel to $E_3$.
\item[(ii)] Only $E''_1$ is parallel to $E_8$, or only $E''_5$ is parallel to $E_2$.
\end{enumerate}

In case of (i), applying an argument like in Case 1, the edges $E''_1, E''_2, E''_3$ and $E''_4$ belong to a parallelogram, a hexagon, another hexagon and a parallelogram, respectively, resulting in the configuration in (b) of Figure~\ref{fig:types}. Similarly, from (ii) we obtain a configuration combinatorially equivalent to the configuration in (c) of Figure~\ref{fig:types}.

\emph{Case 3}, $E$ is type 3. Then a similar argument yields the assertion.
\end{proof}

\begin{rem}\label{rem:typeneighbors}
The following table shows the possible types and/or subtypes of the two neighbors of a type 3/4 or 5 side. cf. The orientation of the side is the same as in Figure~\ref{fig:types}.

\begin{table}[h]
\begin{tabular}{c|c|c}
Type/Subtype & left-hand side & right-hand side\\
 of a side & neighbor & neighbor\\
\hline
5 & 2 or 3a & 2 or 3a\\
4a & 2 or 3a & 2 or 3a\\
4b & 1 or 2 & 2 or 3a\\
3a & 1, or 2 or 3 & 1, or 2 or 3\\
3b & 1, or 2, or 3a & 1, or 2 or 3a\\
3c & 1 & 2 or 3a\\
3d & 1 or 2 & 2 or 3a\\
3e & 1 or 2 & 1 or 2
\end{tabular}
\smallskip
\caption{The types/subtypes of the two neighbors of a type 3/4/5 side of $P$}
\end{table}
\end{rem}

To characterize the combinatorial classes of tilings, we distinguish the following cases; in each case the number in parentheses shows the number of combinatorial classes that belong to that case.

\begin{itemize}
\item[I)] $P$ has at least one type 5 edge (20 configurations).
\item[II)] $P$ has no type 5 edge, but it has some type 4 edges (25 configurations).
\item[III)] $P$ has only type 1 edges (1 configuration).
\item[IV)] $P$ has only type 2 edges (4 configurations).
\item[V)] $P$ has only type 3 edges (2 configurations).
\item[VI)] $P$ has some type 1 and type 2 edges, but no other types (13 configurations).
\item[VII)] $P$ has some type 1 and type 3 edges, but no other types (5 configurations).
\item[VIII)] $P$ has some type 2 and type 3 edges, but no other types (29 configurations).
\item[IX)] $P$ has some type 1,2 and 3 edges, but no other types (12 configurations).
\end{itemize}

We present the proof of Case VIII in detail, and only sketch the proofs of the other cases.
Due to the large number of cases, in the proof we need to rely on labels introduced only in figures. To make these labels well-defined, in each case we keep the following rules:
\begin{itemize}
\item $E_1$ is the bottom horizontal side of $P$;
\item apart from $E_1, \ldots, E_8$, only vertices, edges and faces of $\F$ are labelled;
\item vertices, edges and faces of $\F$ are labelled $p_i$, $E'_i$ and $F_i$, respectively, for some integer value of $i$;
\item labelled vertices are denoted by small circles, labelled edges by dotted lines, and labelled faces always contain their labels.
\end{itemize}

\begin{proof}[Proof of Case VIII]
We have four possibilities:
\begin{itemize}
\item $P$ has exactly one pair of type 3 edges,
\item $P$ has two consecutive pairs of type 3 edges,
\item $P$ has two nonconsecutive pairs of type 3 edges,
\item $P$ has three pairs of type 3 edges.
\end{itemize}

In other words, we may assume that $P$ is of type either 3/2/2/2, or 3/2/2/3, or 3/2/3/2, or 3/2/3/3.
Note that, by Remark~\ref{rem:typeneighbors}, $P$ has no subtype 3c side.

\emph{Subcase VIII.1}, $P$ is type 3/2/2/2.

Before proceeding further, we observe that an elementary consideration shows that if a side of subtype 3d or 3e has two type 2 neighbors, then the faces of $\F$ intersecting at least one of these three sides form a configuration combinatorially equivalent to those in (a) and (b) of Figure~\ref{fig:double3e}.

\begin{figure}[ht]
\includegraphics[width=\textwidth]{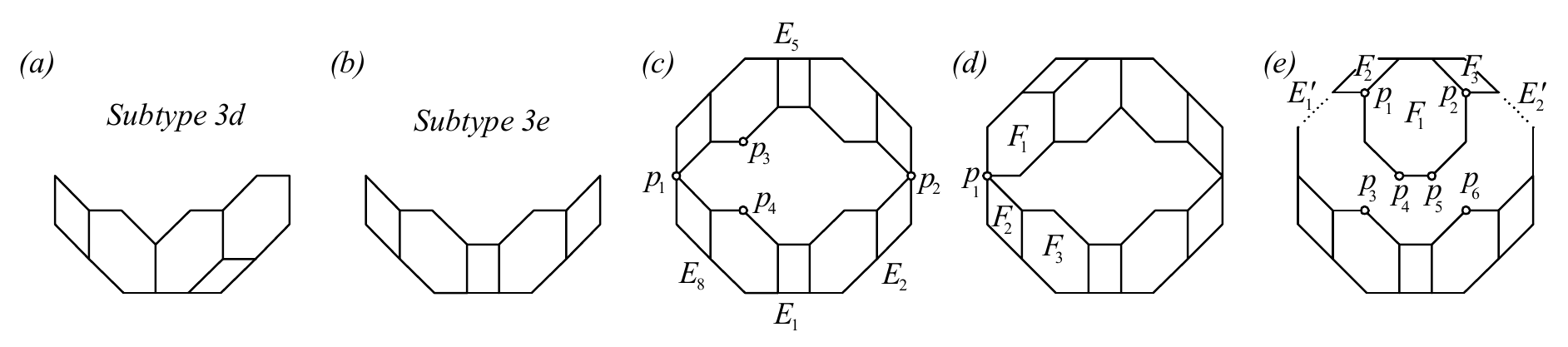}
\caption[]{Illustrations for the proof of Case VIII, part 1}
\label{fig:double3e}
\end{figure}

\emph{Subcase VIII.1.1}, $P$ has a subtype 3e side. Without loss of generality, we may assume that $E_1$ is subtype 3e.

Assume that $E_5$ is also subtype 3e. Then the faces intersecting any of the $E_i$s for $i=1,2,4,5,6,8,$ form a configuration congruent to the one in (c) of Figure~\ref{fig:double3e}. We use the notations of this figure. By Lemma~\ref{lem:norightangle}, besides the two parallelograms shown in (c), $p_1$ belongs to exactly one more face $F_1$ of $\F$, which is either a rectangle, or a hexagon. Consider the case that $F_1$ is a rectangle. Then, as every edge in $\F$ is parallel to a side of $P$, $F_1$ has common edges with two more parallelograms in $\F$, one containing $p_3$, and another one containing $p_4$. The union of $F_1$ with these two parallelograms is a convex hexagon, which contradicts our assumption that $\T$ is irreducible. Thus, $F_1$ is a hexagon. We may repeat the same argument for $p_2$, and then for $p_3$ and $p_4$, which yields the configuration VIII/1 in the Appendix.

Now consider the case that the other type 3 side of $P$ is subtype 3d. Part (d) of Figure~\ref{fig:double3e} shows the combinatorial structure of the faces of $\F$ intersecting some $E_i$ for $i=1,2,4,5,6,8$. Note that the point $p_1$ belongs to three faces: the hexagon $F_1$, the parallelogram $F_2$ and another parallelogram, not shown in the figure, which we denote by $F_4$. Then the union of $F_2$, $F_3$ and $F_4$ is a convex hexagon, which contradicts the irreducibility of $\T$.

Assume that the other type 3 side of $P$ is subtype 3b; we use the notations of (e) of Figure~\ref{fig:double3e}. Observe that the segment $[p_3,p_4]$
is a translate of $E'_1$, and let $L$ denote the line of $[p_3,p_4]$. Consider the class of $E'_1$ (cf. Lemma~\ref{lem:edgeclasses}), and let $S$ be the union of all faces of $\F$ that contain an edge from this class. Then every line parallel to $L$ that intersects $P$ intersects $S$ in a segment of length at least $|p_4-p_3|$.
Thus, $L \cap S = [p_3,p_4]$, from which, since $p_3$ and $p_4$ are vertices of $\F$, it follows that $[p_3,p_4]$ is an edge of $\F$.
We obtain similarly that $[p_5,p_6]$ is an edge of $\F$, implying, by the irreducibility of $\T$, that $[p_3,p_4]$ and $[p_5,p_6]$ are edges of the same hexagon tile of $\F$.
By Lemma~\ref{lem:norightangle}, besides $F_1$ and $F_2$, $p_1$ is the vertex of exactly one more tile, which is either a rectangle or a hexagon. It is easy to see that as $\T$ is irreducible, this tile cannot be a hexagon, which yields also that $E'_1$ belongs to a parallelogram tile. Repeating the same argument for $p_2$, we obtain the configuration VIII/2 in the Appendix.

\begin{figure}[ht]
\includegraphics[width=\textwidth]{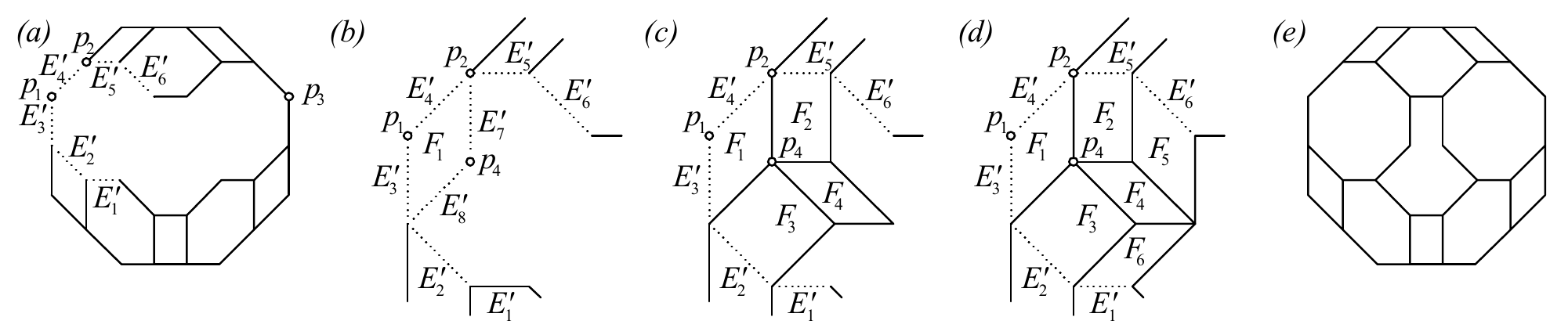}
\caption[]{Illustrations for the proof of Case VIII, part 2}
\label{fig:type3222_fig2}
\end{figure}

Finally, we are left with the case that the other type 3 edge of $\F$ is subtype 3a. We use the notations of (a) in Figure~\ref{fig:type3222_fig2}.
From the irreducibility of $\T$ it follows that the only edges of $\F$ starting at $p_1$ are $E'_3$ and $E'_4$. Thus, $E'_3$ and $E'_4$ belong to the same tile $F_1$. Consider the case that $F_1$ is a parallelogram. Let $E'_7$ and $E'_8$ be the other two sides of $F_1$ such that $E'_7$ is parallel to $E'_3$.
By Lemma~\ref{lem:norightangle}, $E'_5$ and $E'_7$, and $E'_2$ and $E'_8$ belong to the same tiles, which we denote by $F_2$ and $F_3$, respectively; furthermore, $F_2$ and $F_3$ are parallelograms or hexagons. Clearly, not both these tiles are hexagons. If both are parallelograms, then, besides $F_1$, $F_2$ and $F_3$, $p_4$ belongs to one more tile $F_4$, which is a parallelogram (cf. part (c) of Figure~\ref{fig:type3222_fig2}). Then both $E'_1$ and $E'_2$ belong to a parallelogram tile, let these tiles be $F_5$ and $F_6$ (cf. part (d) of Figure~\ref{fig:type3222_fig2}). Since $\bigcup_{i=1}^6 F_i$ is a convex octagon, we have reached a contradiction. If exactly one of $F_2$ and $F_3$ is a hexagon, we may apply a similar argument. Hence, we can assume that the tile $F_1$, containing $E'_3 \cup E'_4$ is not a parallelogram. A similar consideration shows that it is not a hexagon, either, which yields that it is an octagon . By symmetry, we have also that the point $p_3$ belongs to exactly one tile in $\F$, and this tile is an octagon (cf. (e) of Figure~\ref{fig:type3222_fig2}). These observations, and the irreducibility of $\T$ readily yields that in this case the configuration is VIII/3 in the Appendix.

\emph{Subcase VIII.1.2}, $P$ has a side of subtype 3d, but has no side of subtype 3e. Without loss of generality, we may assume that $E_1$ is subtype 3d.
Then the collection of the tiles intersecting $E_i$ for $i=1,2,3,4,8$ is congruent to (a) of Figure~\ref{fig:type3222_fig3}.

If the other type 3 side is subtype 3d, then the collection of the tiles containing an edge on any edge of $P$ is congruent to (b) of Figure~\ref{fig:type3222_fig3}. We use the notations of this figure. By Lemma~\ref{lem:norightangle}, both $P_1$ and $p_2$ belong to exactly three tiles of $\F$: two of them are parallelograms, shown in the figure, and the third ones are rectangles or hexagons. Furthermore, since $\T$ is irreducible, the third tiles are hexagons. This yields VIII/4 in the Appendix.

Consider the case that the other type 3 side is subtype 3b. We use the notations in (c) of Figure~\ref{fig:type3222_fig3}. By Lemma~\ref{lem:norightangle}, there is a rectangle or a hexagon tile that contains both edges $E'_1$ and $E'_2$. On the other hand, since $\T$ is irreducible, this tile is a hexagon. This implies that the sides $E'_3$, $E'_4$ and $E'_5$ belong to parallelograms, each of which has two sides parallel to $E_4$ (cf. (d) of Figure~\ref{fig:type3222_fig3}). Now, similarly like before, there is also a hexagon tile containing $p_1$ (cf. (e) of Figure~\ref{fig:type3222_fig3}), which readily yields the configuration VIII/5 in the Appendix.

\begin{figure}[ht]
\includegraphics[width=\textwidth]{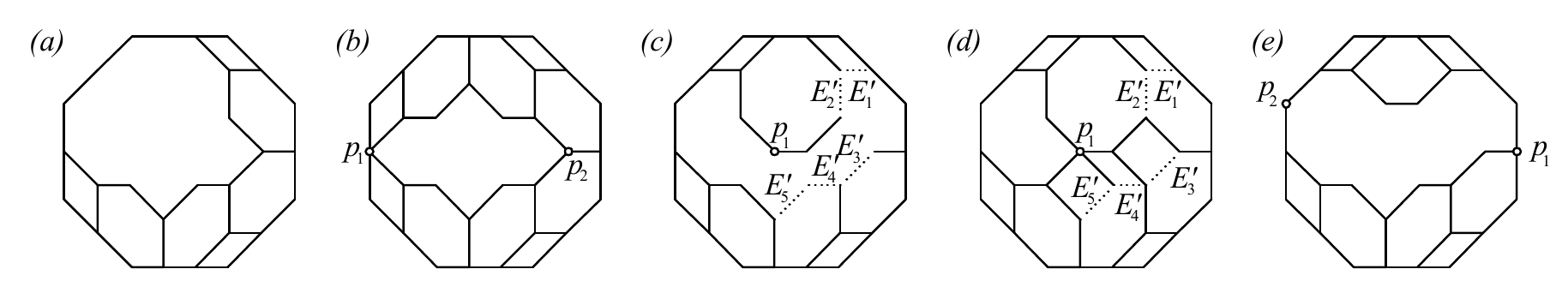}
\caption[]{Illustrations for the proof of Case VIII, part 3}
\label{fig:type3222_fig3}
\end{figure}

Assume that the other type 3 side is subtype 3a. We use the notations in (e) of Figure~\ref{fig:type3222_fig3}. By Lemma~\ref{lem:norightangle} and the irreducibility of $\T$ it follows that $p_1$ belongs to two hexagon tiles. Furthermore,  repeating the argument in the last paragraph of Subcase VIII.1.1, we have that $p_2$ belongs to an octagon tile. Thus, using the fact that every edge in $\F$ is parallel to an edge of $P$, we obtain that the only possible configuration in this case is VIII/6 in the Appendix.

\emph{Subcase VIII.1.3}, $P$ has no side of subtypes 3d or 3e.

Assume that $P$ has two subtype 3a sides. Then we may use the notations in (a) of Figure~\ref{fig:type3222_fig4}. Observe that the segment $[p_1,p_3]$ is vertical, and has unit length. Let $S$ denote the union of the tiles that have vertical sides. Then every vertical line that intersects $S$ intersects it in a segment of at least unit length. Thus, like in Subcase VIII.1.1, we have that $[p_1,p_3]$, and similarly $[p_2,p_4]$, is the union of two vertical edges of $\F$. Thus
$\conv\{ p_1,p_2,p_3,p_4 \}$ is the union of two tiles of $\F$, which contradicts the irreducibility of $\T$.

Consider the case now that $P$ has two subtype 3b sides. We use the notations in (b) of Figure~\ref{fig:type3222_fig4}. By Lemma~\ref{lem:norightangle}, besides the two octagons, each of $p_1$ and $p_2$ belongs to one more face of $\F$, and this face is a parallelogram or a hexagon. Let, for instance, $p_1$ belong to a parallelogram $F_1$. Observe that the distance of the midpoint $p_3$ of $E_7$, and the vertex $p_4$ of $F_1$ farthest from $p_1$, is equal to $\frac{1}{3}$. Let $S$ denote the union of the faces that contain edges from the class of the leftmost edge of $\F$ in $E_1$. Since every horizontal line intersects $S$ in a segment of length at least $\frac{1}{3}$, and $p_3$ and $p_4$ are vertices of $\F$, it follows that $[p_3,p_4]$ is an edge of $\F$. Then, by the irreducibility of $\T$, $[p_3,p_4]$ belongs to two hexagon tiles of $\F$. Similarly, if, for example, $p_2$ belongs to a hexagon $F_2$, then $F_2$ shares an edge with four more parallelogram tiles (cf. (c) of Figure~\ref{fig:type3222_fig4}). Since these two subconfigurations can occur independently, we obtain three more configurations: when both $p_1$ and $p_2$ belong to parallelograms, or when exactly one of them, or when neither. These configurations are shown as VIII/7, VIII/8 and VIII/9 in the Appendix, respectively.

We are left with the case that $P$ has one subtype 3a, and one subtype 3b side. Without loss of generality, we may assume that $E_1$ is subtype 3b, and $E_5$ is subtype 3a. We use the notations of (d) of Figure~\ref{fig:type3222_fig4}. First, similarly like in the last paragraph of Subcase VIII.1.1, we have that the segments $[p_1,p_3]$ and $[p_2,p_4]$ are edges of $\F$. Thus, the undivided part of $P$ consists of two congruent, unconnected regions, which we call \emph{left} and \emph{right} regions such that the left region contains $p_1$. Now, by Lemma~\ref{lem:norightangle}, besides an octagon and a parallelogram, each of $p_5$ and $p_6$ belongs to one additional tile, which we denote by $F_1$ and $F_2$, respectively, and which are rectangles or hexagons.

Consider the case that one of them, say $F_1$, is a rectangle, and let $p_7$ be the vertex of $F_1$ on $E_8$. Then $p_7$ belongs to three parallelogram tiles, and, by the irreducibility of $\T$, the remaining part of the left region is an octagon tile (cf. (e) in Figure~\ref{fig:type3222_fig4}). If one of $F_1$ and $F_2$, say $F_2$, is a hexagon, then let $p_8$ denote the midpont of $E_3$. By Lemma~\ref{lem:norightangle}, $p_8$ belongs to only one more tile $F_3$, and the irreducibility of $\T$ yields that $F_3$ is a hexagon. As every side of $\F$ is parallel to a side of $P$, the remaining part of the right region is divided into two parallelograms. If both $F_1$ and $F_2$ are hexagons, then the resulting configuration is not irreducible, and hence, we have two configurations in this case: in one both $F_1$ and $F_2$ are parallelograms, and in the other one exactly one of them is a hexagon. These configurations are shown as VIII/10 and VIII/11 in the Appendix, respectively.

\begin{figure}[ht]
\includegraphics[width=\textwidth]{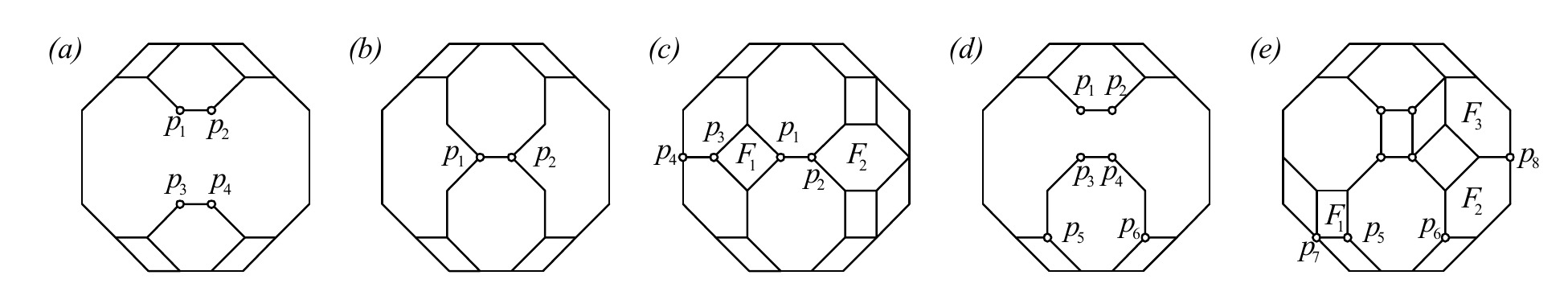}
\caption[]{Illustrations for the proof of Case VIII, part 4}
\label{fig:type3222_fig4}
\end{figure}

\emph{Subcase VIII.2}, $P$ is type 3/2/2/3.

Observe that by Remark~\ref{rem:typeneighbors}, for any two consecutive type 3 sides, at least one of them is subtype 3a, and also that $P$ has no subtype 3c and subtype 3e sides. Thus, without loss of generality, we may assume that $E_8$ is subtype 3a.

\emph{Subcase VIII.2.1}, $P$ has a subtype 3d side. Again, without loss of generality, we may assume that $E_1$ is subtype 3d. Then the tiles intersecting $E_i$ for $i=1,2,8$ are combinatorially equivalent to (a) in Figure~\ref{fig:type3223_fig1}. 

Consider the case that another side of $P$ is also subtype 3d. It is easy to see that the irreducibility of $\T$ yields that this side is $E_4$. Hence, $E_5$ is subtype 3a. For simplicity, we use the notations of part (b) of Figure~\ref{fig:type3223_fig1}. Since the sides of any tile of $\F$ are parallel to a side of $P$, $p_1$ belongs to two parallelogram tiles, one of which is contained in $\inter P$. Let $p_3$ denote the vertex of this parallelogram opposite to $p_1$. Clearly, the two sides of this parallelogram containing $p_3$ belong to two parallelogram tiles (cf. (c) of Figure~\ref{fig:type3223_fig1}). Repeating the argument in the last paragraph of VIII.1.1, we obtain that both $p_2$ and $p_3$ belong to an octagon tile; a contradiction.

Assume that $P$ also has a subtype 3b side. If this side is $E_4$, then the tiles intersecting $E_i$ for $i\neq 6,7$ form the configuration in (d) of Figure~\ref{fig:type3223_fig1}. Clearly, it cannot be completed to an irreducible configuration, which implies that $E_4$ is subtype 3a, and $E_5$ is subtype 3b.
Hence, we may use the notations in (e) of Figure~\ref{fig:type3223_fig1}. By Lemma~\ref{lem:norightangle}, we have that besides those shown in the figure, each of $p_1$ and $p_2$ belongs to one more tile, a parallelogram or a hexagon. By the irreducibility of $\T$, it follows that $p_1$ belongs to a hexagon, whereas $p_2$ belongs to a parallelogram. Then, using Lemma~\ref{lem:norightangle}, the irreducibility of $\T$, and the fact that every side in $\F$ is parallel to a side of $P$, an elementary consideration shows that in this case $\F$ is equivalent to VIII/12 in the Appendix.

\begin{figure}[ht]
\includegraphics[width=\textwidth]{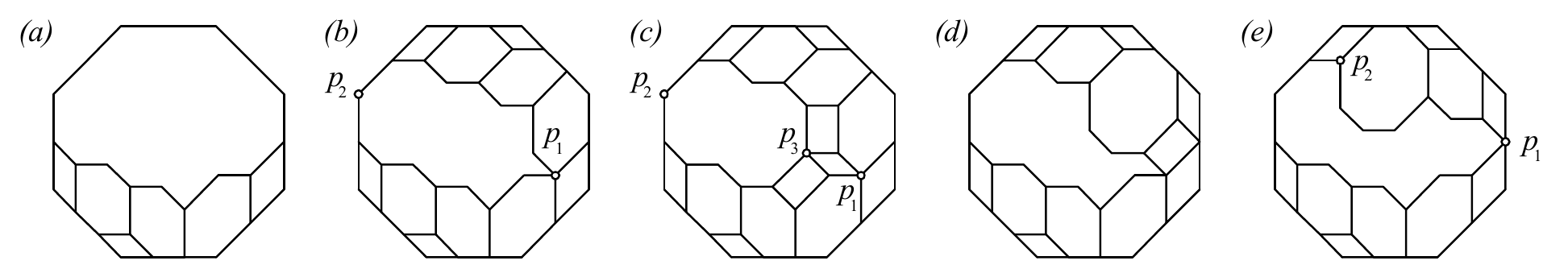}
\caption[]{Illustrations for the proof of Case VIII, part 5}
\label{fig:type3223_fig1}
\end{figure}

Assume now that $P$ has three subtype 3a sides. We use the notations of (a) of Figure~\ref{fig:type3223_fig2}. Observe that, besides those shown in the figure, $p_2$ belongs to one more tile of $\F$, which is a parallelogram. By Lemma~\ref{lem:norightangle}, besides those shown in the figure, $p_1$ belongs to one more tile $F_1$, which is a parallelogram or a hexagon. Furthermore, the irreducibility of $\T$ yields that $F_1$ is a hexagon (cf. (b) of Figure~\ref{fig:type3223_fig2}). Let the vertex of $F_1$, opposite to $p_1$, be denoted by $p_3$. Then the sides of $F_1$ containing $p_3$ belong to two parallelogram tiles, and the resulting configuration is not irreducible; a contradiction.

\emph{Subcase VIII.2.2}, All type 3 sides of $P$ are subtypes 3a or 3b.

\begin{figure}[ht]
\includegraphics[width=\textwidth]{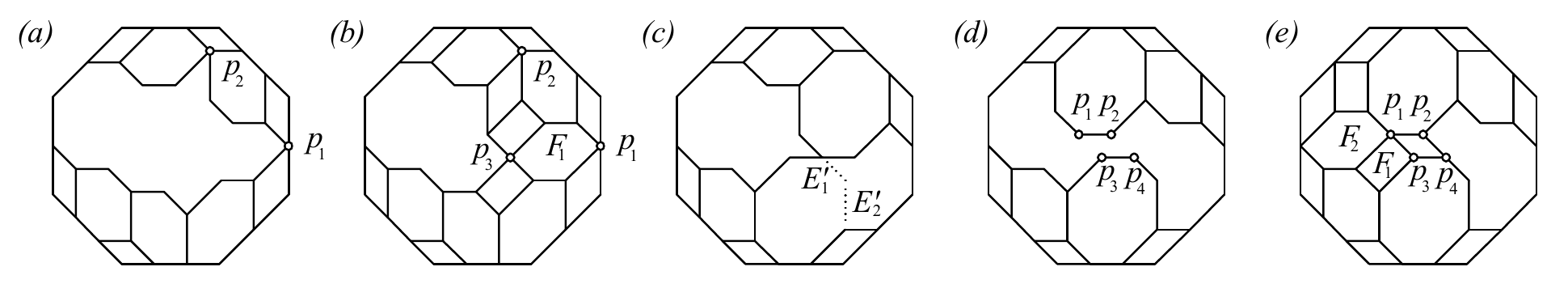}
\caption[]{Illustrations for the proof of Case VIII, part 6}
\label{fig:type3223_fig2}
\end{figure}

Let us assume that $P$ has two subtype 3b sides. First, we consider the case that these two sides are not opposite sides of $P$; let them be, say, $E_1$ and $E_4$. Then we may use the notations of (c) of Figure~\ref{fig:type3223_fig2}. Then, besides the octagon shown in the figure, both $E'_1$ and $E'_2$ belong to a parallelogram tiles of $\F$, which implies that $\T$ is not irreducible; a contradiction. Thus, let $E_1$ and $E_5$ be subtype 3b sides. We use the notations of (d) of Figure~\ref{fig:type3223_fig2}. Note that $| p_1 - p_3 | = \frac{1}{3}$, which yields, similarly like in the third paragraph of Subcase VIII.1.1, that $[p_1,p_3]$ is an edge of $\F$, belonging to the class of the middle edge in $E_8$. Thus, $F_0=\conv \{ p_1,p_2,p_3,p_4 \}$ is a tile in $\F$. Then the uncovered part of $P$ consists of two unconnected, congruent regions. We call the region intersecting $E_7$ \emph{left} region, and the other one \emph{right} region. Then, by Lemma~\ref{lem:norightangle}, $[p_1,p_3]$ belongs to $F_0$ and another parallelogram, which we denote by $F_1$ (cf. (e) of Figure~\ref{fig:type3223_fig2}). Similarly, Lemma~\ref{lem:norightangle} and the irreducibility of $\T$ implies that besides an octagon, $F$ and $F_1$, $p_1$ belongs to a hexagon tile, which we denote by $F_2$. The remainder of the left region is clearly divided into two parallelograms. By symmetry, the dissection of the right region is the reflected copy
of the dissection of the left region, about the center of $P$. This yields VIII/13 in the Appendix.

Next, consider the case that $P$ has exactly one subtype 3b side. Without loss of generality, we assume that $E_1$ is subtype 3b, and use the notations of (a) of Figure~\ref{fig:type3223_fig3}. Then $p_1, p_2$ and $p_3$ belong to a parallelogram tile. By Lemma~\ref{lem:norightangle} and since different tiles do not overlap, $p_3, p_4$ and $p_5$ belong to another parallelogram tile (cf. (b) of Figure~\ref{fig:type3223_fig3}). By the same lemma, the other tile $[p_3,p_5]$ belongs to is a rectangle or a hexagon. On the other hand, in both cases we reach a contradiction with our condition that $\T$ is irreducible.

Finally, assume that $P$ has four type 3a sides. Then the tiles intersecting $E_i$ for $i=1,4,5,8$ are shown in (c) of Figure~\ref{fig:type3223_fig3}. We use the notations of this figure. Similarly like in the last paragraph of Subcase VIII.1.1, we obtain that $p_1$ and $p_2$ belong to octagon tiles (cf. (d) of Figure~\ref{fig:type3223_fig3}). Then, as the sides of every tile are parallel to some sides of $P$, we readily obtain the configuration in VIII/14 in the Appendix.

\begin{figure}[ht]
\includegraphics[width=\textwidth]{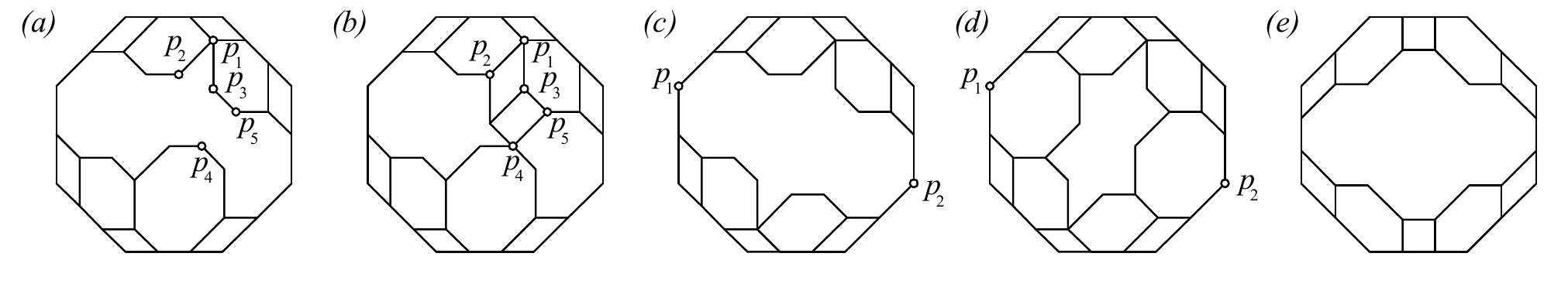}
\caption[]{Illustrations for the proof of Case VIII, part 7}
\label{fig:type3223_fig3}
\end{figure}

\emph{Subcase VIII.3}, $P$ is type 3/2/3/2.

\emph{Subcase VIII.3.1}, $P$ has a subtype 3e side.

Without loss of generality, let $E_1$ be subtype 3e. Observe that in this case $E_3$ and $E_7$ are not subtype 3e. First, we consider the case that $E_5$ is subtype 3e as well (cf. (e) of Figure~\ref{fig:type3223_fig3}). Then the irreducibility of $\T$ implies that $E_3$ and $E_7$ are subtype 3b, and we readily obtain the configuration VIII/15 in the Appendix. Now, let $E_5$ be subtype 3d. By the symmetry of the conditions, we may assume that the common endpoint of $E_5$ and $E_6$ belongs to a hexagon tile. Then, by the irreducibility of $\T$, it follows that $E_7$ is subtype 3b, and $E_5$ is subtype 3d (cf. (a) of Figure~\ref{fig:type3232_fig1}). Applying Lemma~\ref{lem:norightangle} and the fact that $\T$ is irreducible, we obtain VIII/16 in the Appendix.

Let $E_5$ be subtype 3b. Then the tiles intersecting $E_i$ for $i=1,2,5,8$ are congruent to those in (b) of Figure~\ref{fig:type3232_fig1}. We use the notations of this figure. By Lemma~\ref{lem:norightangle}, besides those shown in this figure, each of $p_1$ and $p_2$ belongs to one more tile, which is a parallelogram or a hexagon. If both are hexagons, then the irreducibility of $\T$ implies that both $E_3$ and $E_7$ are subtype 3d (cf. (c) in Figure~\ref{fig:type3232_fig1}), and we obtain the configuration VIII/17 in the Appendix. If both are parallelograms, then we obtain similarly that both $E_3$ and $E_7$ are subtype 3b (cf. (d) in Figure~\ref{fig:type3232_fig1}), and the configuration is equivalent to VIII/18 in the Appendix. Finally, let exactly one of them, say the one containing $p_1$, be a parallelogram, and the other one, containing $p_2$, be a hexagon. Then  $E_3$ is subtype 3d, and $E_7$ is subtype 3b, and the irreducibility of $\T$ implies that the configuration is equivalent to VIII/19 in the Appendix.

\begin{figure}[ht]
\includegraphics[width=\textwidth]{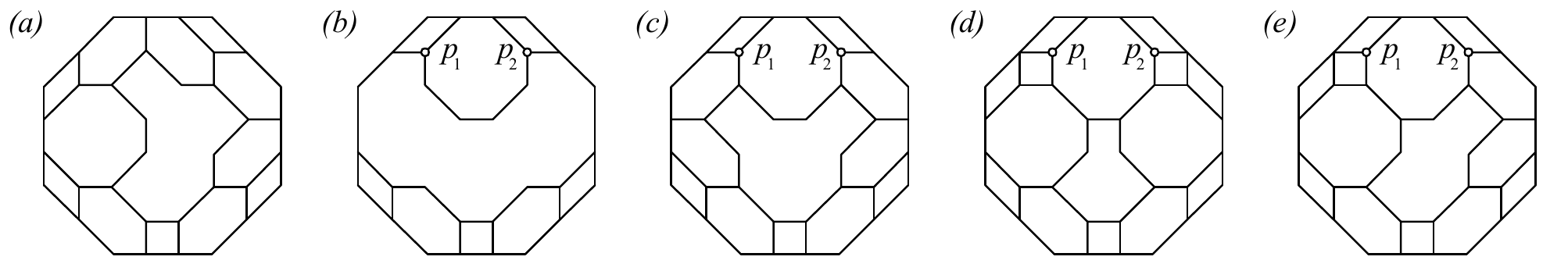}
\caption[]{Illustrations for the proof of Case VIII, part 8}
\label{fig:type3232_fig1}
\end{figure}

\emph{Subcase VIII.3.2}, $P$ has no subtype 3e side.

First, consider the case, that there is a subtype 3d side of $P$, say $E_1$. Without loss of generality, we may assume that $E_1$ is subtype 3d, and the tiles intersecting $E_i$ for $i=1,2,8$ form a configuration equivalent to the one in (a) of Figure~\ref{fig:type3232_fig2}. Since $P$ has no subtype 3e side, $E_3$ is subtype 3d as well, which yields that $E_5$ and $E_7$ are also subtype 3d. Using Lemma~\ref{lem:norightangle}, it follows that the configuration is equivalent to VIII/20 in the Appendix.

From now on, we may assume that every subtype 3 side of $P$ is subtype 3a or subtype 3b. Consider first the case that $P$ has a subtype 3b side. Without loss of generality, we may assume that $E_1$ is subtype 3b. By the irreducibility of $\T$, we obtain that then both $E_3$ and $E_5$ are subtype 3b, which, applying this observation also for $E_3$, yields that every type 3 side of $P$ is subtype 3b (cf. (c) of Figure~\ref{fig:type3232_fig2}). Thus, $\F$ is equivalent to VIII/21.
We are left with the case that every type 3 side of $P$ is subtype 3a (cf. (d) of Figure~\ref{fig:type3232_fig2}). Then, by Lemma~\ref{lem:norightangle} and the irreducibility of $\T$, the midpoint of every $E_i$, where $i=2,4,6,8$, belongs to two parallelograms and a hexagon (cf. (e) of Figure~\ref{fig:type3232_fig2}).
Note that the remaining part of $P$ cannot be dissected into an irreducible tiling; hence, there is no suitable configuration in this case.

\begin{figure}[ht]
\includegraphics[width=\textwidth]{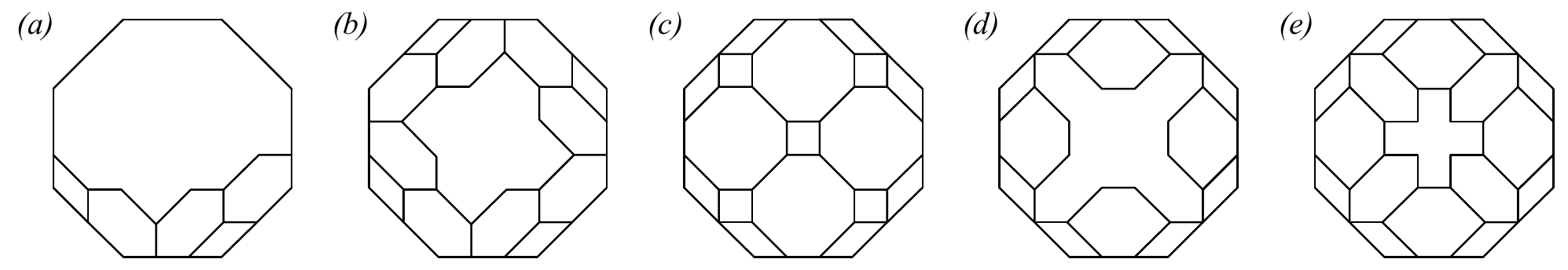}
\caption[]{Illustrations for the proof of Case VIII, part 9}
\label{fig:type3232_fig2}
\end{figure}

\emph{Subcase VIII.4}, $P$ is type 3/3/2/3.
First, note the following consequences of Remark~\ref{rem:typeneighbors}.
\begin{itemize}
\item since every type 3 side of $P$ has a type 3 side neighbor, $P$ has no subtype 3e side;
\item consequently, every side of $P$ is subtype 3a, 3b or 3d.
\item if two type 3 sides of $P$ are consecutive, then at least one of them is subtype 3a;
\item If $E_1$ or $E_5$ is not subtype 3a, then it is subtype 3b, and its two neighbors are subtype 3a.
\end{itemize} 

\emph{Subcase VIII.4.1}, if $E_1$ or $E_5$ is subtype 3b.

First, assume that both $E_1$ and $E_5$ are subtype 3b. Then the tiles intersecting $E_i$ for $i \neq 3,7$ are equivalent to the configuration in (a) of Figure~\ref{fig:type3323_fig1}. By Lemma~\ref{lem:norightangle} and the irreducibility of $\T$, it immediately follows that in this case $\F$ is equivalent to VIII/22 in the Appendix.

Now, consider the case that either $E_1$ or $E_5$, say $E_1$ is subtype 3b. Then $E_8$, $E_2$ and $E_5$ are subtype 3a (cf. (b) of Figure~\ref{fig:type3323_fig1}). Note that $E_4$ and $E_6$ are subtype 3a, 3b or 3d. Assume that at least one of them, say $E_4$, is subtype 3a (cf. (c) of Figure~\ref{fig:type3323_fig1}). We use the notations in this figure. By Lemma~\ref{lem:norightangle}, besides those shown in the figure, $p_1$ belongs to one more tile $F_1$, which, as $\T$ is irreducible, is a hexagon. The two sides of $\F_1$ of shown in (c) of Figure~\ref{fig:type3323_fig1} belong also to two rectangles. Since, clearly, $p_2$ belongs to two parallelogram tiles (cf. (d) of Figure~\ref{fig:type3323_fig1}), the configuration is not irreducible; a contradiction. We may apply a similar argument if at least one of $E_4$ and $E_6$ is subtype 3b. Thus, it suffices to consider the case that both $E_4$ and $E_6$ are subtype 3d; the tiles intersecting any side of $P$ are shown in (e) of Figure~\ref{fig:type3323_fig1}. Then the irreducibility of $\T$ immediately implies that $\F$ is equivalent to VIII/23 in the Appendix.

\begin{figure}[ht]
\includegraphics[width=\textwidth]{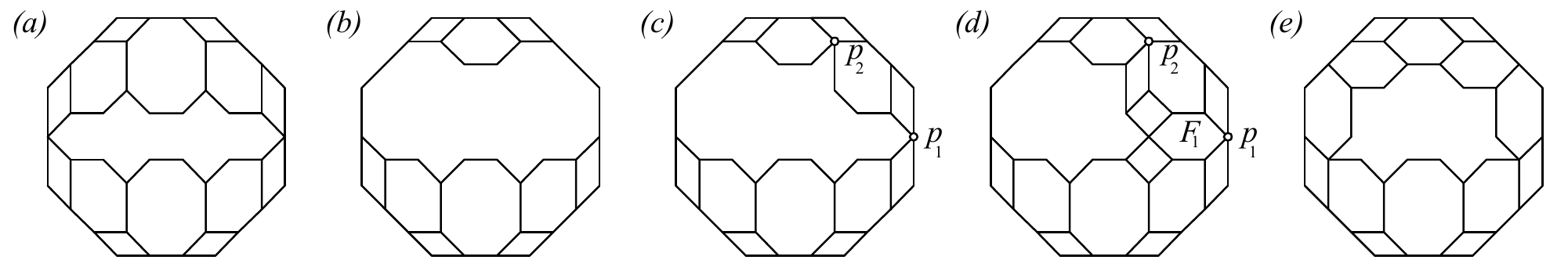}
\caption[]{Illustrations for the proof of Case VIII, part 10}
\label{fig:type3323_fig1}
\end{figure}

\emph{Subcase VIII.4.2}, both $E_1$ and $E_5$ are subtype 3a. Then the neighbors of $E_1$ and $E_5$ are subtype 3a, 3b or 3d.
First, we show that none of them are subtype 3a.

Assume that, say, $E_8$ is subtype 3a. If $E_3$ is subtype 3d, then the tiles intersecting $E_1$ for $i=1,5,6,7,8$ are equivalent to those in (a) of Figure~\ref{fig:type3323_fig2}, and we may use the notations in this figure. Note that, as the sides of every tile are parallel to some sides of $P$, besides those shown in (a), each of $p_1$ and $p_2$ belongs to one more tile, which is a parallelogram in both cases. Clearly, there is a rectangle tile adjacent to both these parallelograms (cf. (b) of Figure~\ref{fig:type3323_fig2}), and the resulting configuration is not irreducible; a contradiction. If $E_3$ is subtype 3b or 3a, we may apply a similar argument. Thus, we have shown that all neighbors of $E_1$ and $E_5$ are subtype 3b or 3d.

Note that the previous consideration shows more: we obtained that if, for any of the pairs $\{E_2,E_4 \}$ and $\{ E_6,E_8\}$, an element of the pair is subtype 3d, then the other one is subtype 3b. Hence, in particular, among the sides of $P$, there are at most two subtype 3d sides.

Consider the case that $P$ has exactly two subtype 3d sides. Then, without loss of generality, let $E_8$ be subtype 3d, which yields that $E_6$ is subtype 3b. If the other subtype 3d side is $E_2$, then the tiles intersecting a side of $P$ are equivalent to those in (c) of Figure~\ref{fig:type3323_fig2}, and the irreducibility of $\T$ implies that $\F$ is equivalent to VIII/24 in the Appendix. If the other subtype 3d side is $E_4$, then 
the tiles intersecting a side of $P$ are equivalent to those in (d) of Figure~\ref{fig:type3323_fig2}. In this case a short case analysis shows that $\F$ is equivalent to either VIII/25 or VIII/26 in the Appendix.

Assume that $P$ has exactly one subtype 3d side. Without loss of generality, let this side be $E_8$. Then the tiles intersecting a side of $P$ are equivalent to those in (e) of Figure~\ref{fig:type3323_fig2}. Thus, Lemma~\ref{lem:norightangle} and the irreducibility of $\T$ yields that $\F$ is equivalent to VIII/27 or VIII/28 in the Appendix. If $P$ has no subtype 3d side, then, clearly, $\F$ is equivalent to VIII/29 in the Appendix.

\begin{figure}[ht]
\includegraphics[width=\textwidth]{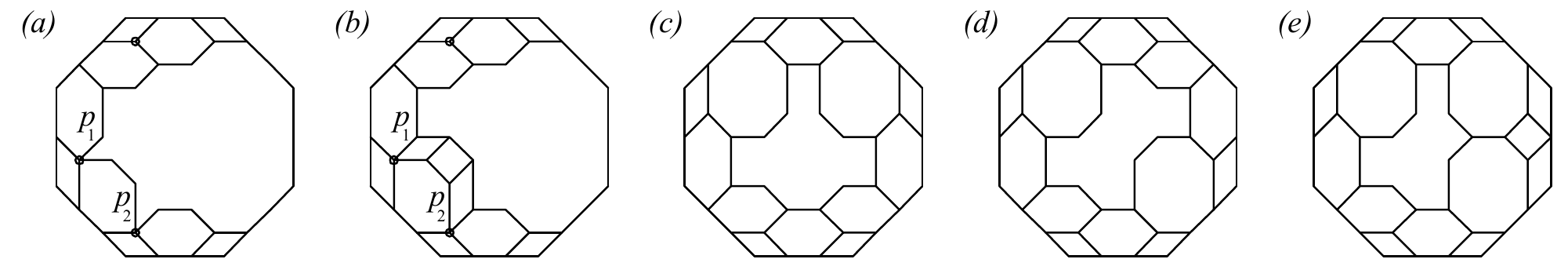}
\caption[]{Illustrations for the proof of Case VIII, part 11}
\label{fig:type3323_fig2}
\end{figure}
\end{proof}

\begin{proof}[Sketch of the proof of the other cases]

First, we deal with Case I. Without loss of generality, we may assume that $E_1$ is type 5. Applying Remark~\ref{rem:typeneighbors}, it follows that $P$ has at most two pairs of type 5 sides. If $P$ has exactly two pairs of type 5 sides, then, by the irreducibility of $\T$, the other two pairs are type 3, and the tiles intersecting any side of $P$ form a configuration equivalent to the one in (a) of Figure~\ref{fig:case1}. A case analyis similar to that in the proof of Case VIII yields four possible configurations in this case.

Assume that the only type 5 sides of $P$ are $E_1$ and $E_5$. Then the neighbors of these sides are type 2 or 3. Let $E_3$ and $E_7$ be type 4. Then, by Lemma~\ref{lem:norightangle} and the irreducibility of $\T$ it follows that either both $E_3$ and $E_7$ are subtype 4a, which yields that $P$ is of type 5/3/4/3,
or both $E_3$ and $E_7$ are subtype 4b, and then $P$ is type 5/2/4/3 or (equivalently) type 5/3/4/2. In these cases, the tiles intersecting a side of $P$ are shown in (b) and (c) of Figure~\ref{fig:case1}. We obtain 8 and 1 possible configurations in the two cases, respectively.

\begin{figure}[ht]
\includegraphics[width=0.55\textwidth]{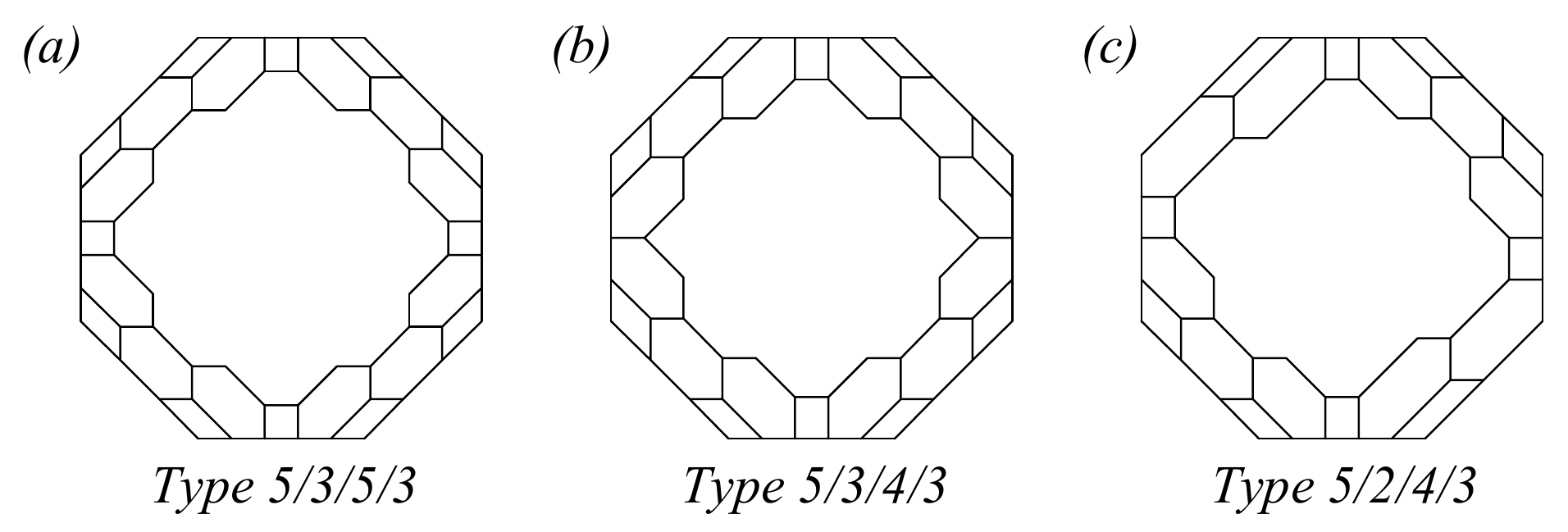}
\caption[]{Illustrations for the proof of Case I}
\label{fig:case1}
\end{figure}

If $E_3$ and $E_7$ are type 3 sides, then, without loss of generality, we may assume that the type of $P$ is either 5/3/3/3, or 5/2/3/3, or 5/2/3/2. Then, applying Remark~\ref{rem:typeneighbors}, we have that the subtypes of $E_3$ and $E_7$ are 3b in the first, 3d in the second, and 3e in the third case.
A careful examination in these cases yields 3,1 and 1 possible configurations, respectively. We obtain two more configurations under the condition that $E_3$ and $E_7$ are type 2. Finally, since the fact that $E_1$ and $E_5$ are type 5 yields that there are at least two classes of vertical sides in $\F$, it follows that $E_3$ and $E_7$ are not type 1.

In Case II, the proof is based on a similar classification scheme, using also the subtypes of the type 4 sides of $P$.

In Cases V, VII and IX, we distinguish cases based on the possible types of the sides of $P$, and also the subtypes of the type 3 sides.
Finally, if all sides of $P$ are type 1 or 2 (i.e. in Cases III, IV and VI), we rely on simple geometric observations, and use tools like in the proof of Case VIII.
\end{proof}

\section{Remarks and questions}\label{sec:remarks}

The problem of enumerating the irreducible, edge-to-edge decompositions of a $(2k)$-gon leads to the following question.

\begin{ques}
Can the irreducible, edge-to-edge decompositions of a $(2k+2)$-gon be generated from the irreducible, edge-to-edge decompositions of a $(2k)$-gon?
\end{ques}

We remark that the obvious way: removing or adding a class of edges in the sense of Lemma~\ref{lem:edgeclasses} does not preserve irreducibility in either direction.

\begin{defn}
For any $k \geq 2$, let $a_k$ denote the number of combinatorial classes of the irreducible, edge-to-edge decompositions of a centrally symmetric convex $(2k)$-gon into centrally symmetric convex pieces.
\end{defn}

Clearly, $a_2 =0$. From \cite{GHA97}, it follows that $a_3=6$. Our result states that $a_4=111$.

\begin{rem}
By Stirling's formula, the estimate in Theorem~\ref{thm:upperbound} is asympotically equal to $\frac{8}{27} \cdot \sqrt{\frac{6}{\pi}} \cdot \frac{1}{N^{\frac{3}{2}}} \cdot \left( \frac{256}{27} \right)^N$, where
$N = \lfloor \frac{2k^3(2k-3)^2}{\left( 2-\sqrt{2} \right) \pi^2 } \rfloor$.
\end{rem}

\begin{ques}
What is the correct magnitude of $a_k$, as a function of $k$? In particular, is $a_k$ bounded from below by a quantity exponential in $k$?
\end{ques}

The rapid increase in the value of $a_k$ for the initial values of $k$ might be an indication that the answer to our last question is positive.

\textbf{Acknowledgments}. The authors gratefully acknowledge the help of M. L\'angi in preparing the figures, and thank an unknown referee for his/her advice on how to improve the bound in Theorem~\ref{thm:upperbound}.

\section{Appendix}

The following tables show representatives of the $111$ combinatorial classes of irreducible, edge-to-edge decompositions of a centrally symmetric octagon. 
The type of a configuration is $a/b/c/d$, if the number of the edges of the partition on four edges of the octagon in counterclockwise direction, starting with the bottom horizontal edge is $a,b,c$ and $d$, respectively. Note that opposite edges of the octagon consist of the same number of edges of the decomposition.
The symbols $\#4=A, \#6=B, \#8=C$ mean that the corresponding configuration consists of $A$ parallelogram, $B$ hexagon, and $C$ octagon tiles.

\begin{figure}[ht]
\includegraphics[width=\textwidth]{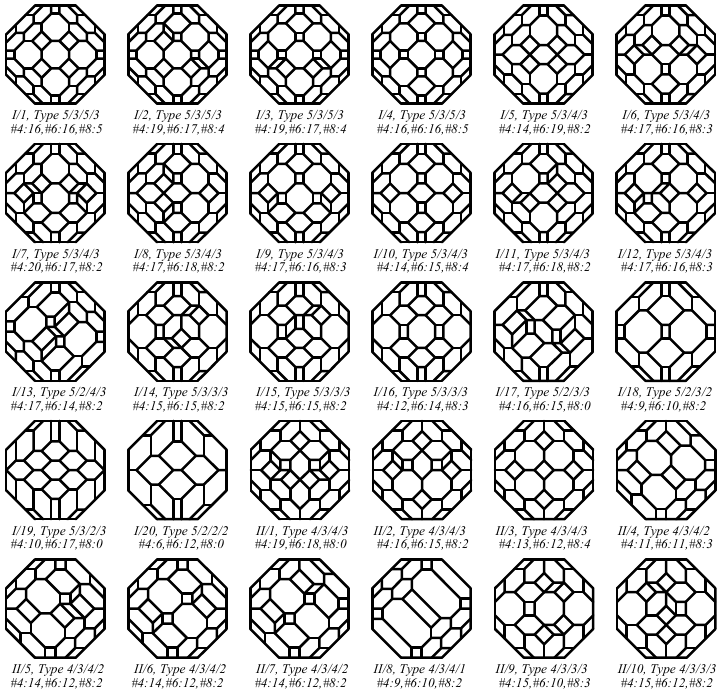}
\end{figure}

\pagebreak

\begin{figure}[ht]
\includegraphics[width=\textwidth]{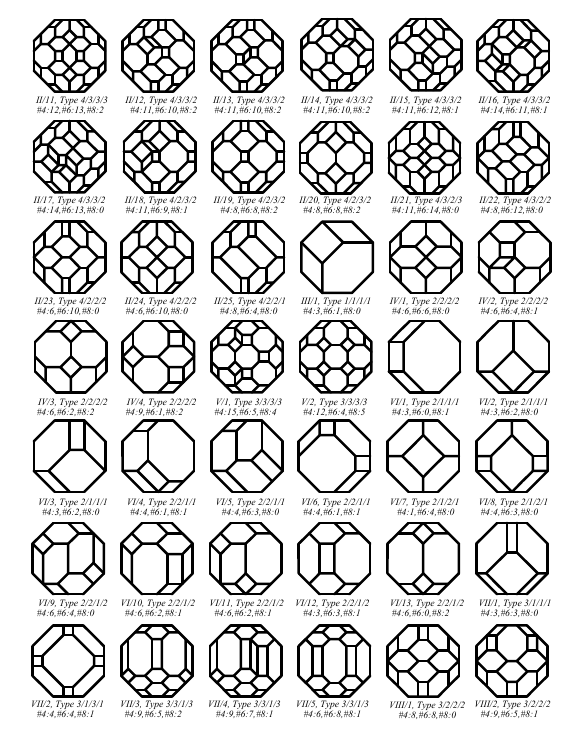}
\end{figure}

\pagebreak

\begin{figure}[ht!]
\includegraphics[width=\textwidth]{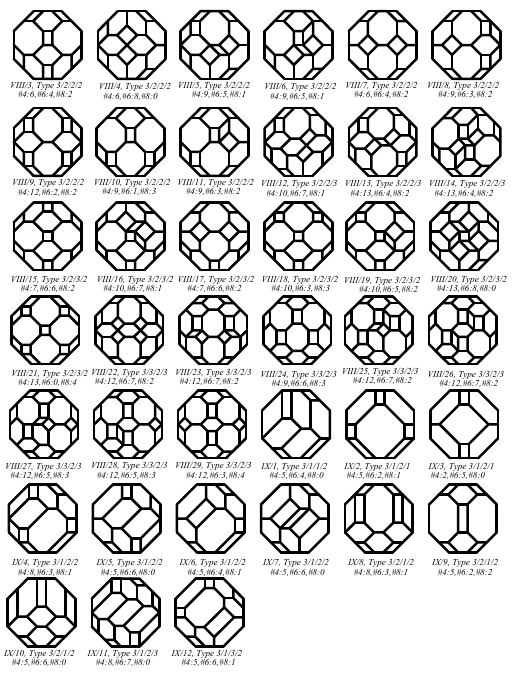}
\end{figure}

\end{document}